\documentclass[reqno]{article}
\usepackage{amsmath,amsfonts,amssymb,amsthm}

\renewcommand{\le}{\leqslant}
\renewcommand{\ge}{\geqslant}
\renewcommand{\rho}{\varrho}

\newtheorem{theorem}{Theorem}
\newtheorem{lemma}[theorem]{Lemma}

\theoremstyle{definition}

\newcommand\eps{\varepsilon}

\renewcommand\Pr{{\mathbb P}}
\newcommand\E{{\mathbb E}}

\newcommand\cc{{\mathrm{c}}}
\newcommand\op{o_{\mathrm{p}}}
\newcommand\Op{O_{\mathrm{p}}}

\newcommand\cA{\mathcal{A}}

\newcommand\cD{\mathcal{D}}
\newcommand\cS{\mathcal{S}}
\newcommand\cW{\mathcal{W}}
\newcommand\cT{\mathcal{T}}
\newcommand\bp[1]{{\cT_{#1}}}

\begin{document}
\title{A simple branching process approach to the phase transition in $G_{n,p}$}
\author{B\'ela Bollob\'as%
\thanks{Department of Pure Mathematics and Mathematical Statistics,
Wilberforce Road, Cambridge CB3 0WB, UK and
Department of Mathematical Sciences, University of Memphis, Memphis TN 38152, USA.
E-mail: {\tt b.bollobas@dpmms.cam.ac.uk}.}
\thanks{Research supported in part by NSF grant DMS-0906634,
ARO grant W911NF-06-1-0076, DARPA grant 9060-200241 CLIN 01,
and University of Memphis FedEx Institute of Technology grant UMF-20953}
\and Oliver Riordan%
\thanks{Mathematical Institute, University of Oxford, 24--29 St Giles', Oxford OX1 3LB, UK.
E-mail: {\tt riordan@maths.ox.ac.uk}.}}
\date{July 24, 2012}
\maketitle

\begin{abstract}
It is well known that the branching process approach to the study of the random graph
$G_{n,p}$ gives a very simple way of understanding the size of the giant component
when it is fairly large (of order $\Theta(n)$). Here we show that a variant of this approach
works all the way down to the phase transition: we use branching process
arguments to give a simple new derivation of the asymptotic size
of the largest component whenever $(np-1)^3n\to\infty$.
\end{abstract}

\section{Introduction}

Our aim in this note is to show how basic results about the survival probability
of branching processes can be used to give an essentially best possible result
about the emergence of the giant component in $G_{n,p}$, the random graph with
vertex set $[n]=\{1,2,\ldots,n\}$ in which each edge is present independently
with probability $p$.
In 1959, Erd\H os and R\'enyi~\cite{ER_evol} showed that if we take
$p=p(n)=c/n$ where $c$ is constant, then there is a `phase transition' at $c=1$.
We write $L_1(G)$ for the maximal number of vertices in a
component of a graph $G$. Also, as usual, we say that an event holds \emph{with high
probability} or \emph{whp} if its probability tends to $1$ as $n\to\infty$.
Erd\H os and R\'enyi showed that, whp, if $c<1$ then $L_1(G_{n,c/n})$ is of
logarithmic order, if $c=1$ it is of order $n^{2/3}$, while if $c>1$ then there
is a unique `giant' component containing $\Theta(n)$ vertices, while the second
largest component is much smaller.

In 1984, Bollob\'as~\cite{BB_evol} noticed that this is only the starting point,
and an interesting question remains: what does the component structure of
$G_{n,p}$ look like for
$p=(1+\eps)/n$, where $\eps=\eps(n)\to 0$? He and {\L}uczak~\cite{Luczak_near}
showed that if $\eps=O(n^{-1/3})$ then $G_{n,p}$ behaves in a qualitatively similar way
to $G_{n,1/n}$; this range of $p$ is now called the \emph{scaling window}
or \emph{critical window} of the
phase transition. The range $\eps^3n\to\infty$ is the \emph{supercritical}
regime, characterized by the fact that there is whp a unique `giant' component
that is much larger than the second largest component. The range $\eps^3n\to -\infty$
is the \emph{subcritical} regime.

In this paper we are interested in the size of the giant component
as it emerges. Thus we consider the (weakly) supercritical regime
where $p=p(n) = (1+\eps)/n$, with $\eps=\eps(n)$ satisfying
\begin{equation}\label{a1}
 \eps\to 0 \hbox{\qquad and \qquad} \eps^3n\to\infty\hbox{\qquad as }n\to\infty.
\end{equation}
Our aim here is to use branching processes
to give a very simple new proof of the following
result, originally due to Bollob\'as~\cite{BB_evol}
(with a mild extra assumption) and {\L}uczak~\cite{Luczak_near}.
\begin{theorem}\label{th1}
Under the assumption \eqref{a1} we have
\[
 L_1(G_{n,p}) = (2+\op(1))\eps n.
\]
\end{theorem}
Here $\op(1)$ denotes a quantity that tends to 0 in probability:
the statement is that for any fixed $\delta>0$, $L_1(G_{n,p})$
is in the range $(2\pm\delta)\eps n$ with probability tending
to $1$ as $n\to\infty$.

Since the original papers~\cite{BB_evol,Luczak_near} (which in fact gave
a more precise bound than that above), many different
proofs of many forms of Theorem~\ref{th1} have been given. For example, Nachmias
and Peres~\cite{NP_giant} used martingale methods to reprove the result as stated
here. Pittel and Wormald~\cite{PWio} used counting methods to prove an even more
precise result; a simpler martingale proof of (part of) their result is given
in~\cite{BR_walk}.  A proof of Theorem~\ref{th1} combining tree counting and
branching process arguments appears in~\cite{rg_bp}. More recently, Krivelevich
and Sudakov~\cite{KS} gave a very simple proof of a variant of Theorem~\ref{th1}
which is even weaker than the original Erd\H{o}s--R\'enyi result:
$\eps>0$ is taken to be constant, and the size of the giant component is determined
only up to a constant factor.

\section{Branching process preliminaries}

Let us start by recalling some basic concepts and results.
The \emph{Galton--Watson} branching process with offspring
distribution $Z$ is the random rooted tree constructed as follows:
start with a single root vertex in generation 0. Each vertex in
generation $t$ has a random number of children in generation $t+1$,
with distribution $Z$. The numbers of children are independent of each
other and of the history. It is well known and easy to check that
if $\E[Z]>1$, then the process \emph{survives} (is infinite) with
probability $\rho$ the unique solution in $(0,1]$ to $1-\rho=f_Z(1-\rho)$,
where $f_Z$ is the probability generating function of $Z$.
When $\E[Z]<1$, the expectation of the total
number of vertices in the branching process is
\begin{equation}\label{totsize}
 1+\E[Z]+\E[Z]^2+\cdots= \frac{1}{1-\E[Z]},
\end{equation}
and in particular the survival probability is $0$.

Let us write $\bp{n,p}$
for the {\em binomial branching process} with parameters $n$ and $p$,
i.e., for the branching process as above with offspring distribution ${\rm Bi}(n,p)$.
Since the generating function of ${\rm Bi}(n,p)$ satisfies
\[
 f(x)=\sum_{k=0}^n \binom{n}{k} p^k (1-p)^{n-k} x^k=\big(1-p(1-x)\big)^n,
\]
when $np>1$ the survival probability
$\rho=\rho_{n,p}$ satisfies
\[
 1-\rho=(1-p \rho)^n.
\]
From this it is easy to check that if $\eps=np-1 \to 0$ with $\eps >0$ then
\begin{equation}\label{2eps}
\rho \sim 2\eps.
\end{equation}

Conditioning on a suitable branching process dying out (i.e., having finite total size)
one obtains another branching process, called the \emph{dual branching process}.
In the binomial case, one way to see
this is to think of $\bp{n,p}$ as a random subgraph of the infinite $n$-ary rooted tree $\bp{n,1}$
obtained
by including each edge independently with probability $p$, and retaining only the component
of the root. For a vertex of $\bp{n,1}$
in generation 1 there are three possibilities: it may (i) be \emph{absent}, i.e.,
not joined to the root, (ii) \emph{survive}, i.e., be joined to the root
and have infinitely many descendents, or (iii) \emph{die out}. The probabilities
of these events are $1-p$, $p\rho$ and $p(1-\rho)$, respectively.
Let $\cD$ denote the event that the process $\bp{n,p}$ dies out, i.e., the total population
is finite. Since $\cD$ happens if and only if every vertex of $\bp{n,1}$ in generation 1 is absent
or dies out, the conditional distribution of $\bp{n,p}$ given $\cD$
is the unconditional distribution of $\bp{n,\pi}$, with $\pi= p(1-\rho) / (1-p\rho)$.
Thus the dual of $\bp{n,p}$ is $\bp{n,\pi}$.

Note that when $np-1=\eps\to 0$, then
\[
 1-n\pi = \frac{1-p\rho-np+np\rho}{1-p\rho} \sim np\rho - (np-1) -p\rho \sim \eps.
\]
Hence the mean number of offspring in the dual process $\bp{n,\pi}$ is $1-(1+o(1))\eps$,
and from \eqref{totsize} its expected size total is $(1+o(1))\eps^{-1}$.
Writing $\cS=\cD^\cc$ for the event that $\bp{n,p}$ \emph{survives} (is infinite),
and $|\bp{n,p}|$ for its total size (number of vertices),
it follows that for any integer $L=L(n)$ we have
\begin{eqnarray}
 \Pr(|\bp{n,p}|\ge L) &=& \Pr(\cS) + \Pr(\cD)\Pr(|\bp{n,\pi}|\ge L) \nonumber \\
 &\le& \Pr(\cS) + \Pr(|\bp{n,\pi}|\ge L) \nonumber \\
 &\le& (1+o(1))(2\eps + 1/(\eps L)), \label{Tup}
\end{eqnarray}
with the second inequality following from Markov's inequality.

We shall use one further property of $\bp{n,p}$, which can be proved in a number
of simple ways. Suppose, as above, that $\eps=np-1\to 0$, and let $M=M(n)$ satisfy
$\eps M\to\infty$. Let $w(\cT)$ denote the \emph{width} of a rooted tree $\cT$, i.e.,
the maximum (supremum) of the sizes of the generations. Then
\begin{equation}\label{unwide}
 \Pr\bigl( \{w(\bp{n,p})\ge M \}\cap \cD\bigr) = o(\eps).
\end{equation}
To see this, consider testing whether the event $\cW_M=\{w(\bp{n,p})\ge M\}$
holds by constructing $\bp{n,p}$
generation by generation, stopping at the first (if any) of size at least $M$.
If such a generation exists then (since the descendents of each vertex in this generation
form independent copies of $\bp{n,p}$), the conditional probability
that the process dies out is at most $(1-\rho)^M\le e^{-\rho M}\to 0$. Hence
\begin{equation}\label{wdie}
 \Pr(\cD\mid \cW_M)=o(1).
\end{equation}
Thus
\[
 \Pr(\cW_M) \sim \Pr(\cS\cap \cW_M) \le \Pr(\cS) \sim 2\eps,
\]
which with \eqref{wdie} gives \eqref{unwide}.

\section{Application to $G_{n,p}$}

The binomial branching process is intimately connected to the component exploration
process in $G_{n,p}$. Given a vertex $v$ of $G_{n,p}$, let $C_v$ denote the component of
$G_{n,p}$ containing $v$, and let $\cT_v$ be the random tree obtained by
exploring this component by breadth-first search. In other words, starting with $v$, find
all its neighbours, $v_1, \dots , v_{\ell}$, say, next find all the neighbours of $v_1$
different from the vertices found so far, then the new neighbours of $v_2$, and so on,
ending the second stage with the new neighbours of $v_{\ell}$. The third stage consists of
finding all the new neighbours of the vertices found in the second stage, and so
on. Eventually we build a tree $\cT_v$, which is a spanning tree of $C_v$.

Note that our notation suppresses the fact that the distributions of $\cT_v$ and of $C_v$ 
depend on $n$ and $p$.
In the next lemma, as usual, $|H|$ denotes the total number of vertices in a graph $H$.

\begin{lemma}\label{couple}
{\rm (i)} For any $n$ and $p$, the random rooted trees $\cT_v$ and
$\bp{n,p}$ may be coupled so that $\cT_v\subset \bp{n,p}$.

{\rm (ii)} For any $n$, $k$ and $p$ there is a coupling of the integer-valued
random variables $|C_v|$ and $|\bp{n-k,p}|$ so that either
$|C_v|\ge |\bp{n-k,p}|$ or both are at least $k$.
\end{lemma}

\begin{proof}
For the first statement we simply generate $\cT_v$ and $\bp{n,p}$
together, always adding fictitious vertices to the vertex set of $G_{n,p}$
for the branching process to take from, so that in each step a vertex
has $n$ potential new neighbours (some fictitious) each of which
it is joined to with probability $p$.
All the descendants of the fictitious vertices are themselves fictitious.

To prove (ii) we slightly modify the exploration, to couple
a tree $\cT_v'$ contained within $C_v$ with $\bp{n-k,p}$ such that one of two
alternatives holds: either $\cT_v' \supset \bp{n-k,p}$, or else both
$\cT_v'$ and $\bp{n-k,p}$ have at least $k$ vertices. Indeed, construct
$\cT_v'$ exactly as $\cT_v$, except that at each step at the start
of which we have not yet reached more than $k$ vertices,
we test for edges from the current vertex to exactly $n-k$ potential new
neighbours. Since $|C_v|\ge |\cT_v'|$, this coupling gives the result.
\end{proof}

From now on we take $p=p(n)=(1+\eps)/n$, where $\eps=\eps(n)$ satisfies \eqref{a1}.
We start by using the two couplings described above to give bounds on the expected number of vertices
in large components. In both lemmas, $N_{[L,n]}$ denotes
the number of vertices of $G_{n,p}$ in components with between $L$ and $n$ vertices (inclusive);
$\Pr_{n,p}$ and $\E_{n,p}$ denote the probability measure and expectation associated
to $G_{n,p}$.
\begin{lemma}\label{elb}
Suppose that $L=L(n)=o(\eps n)$.
Then
$\Pr_{n,p}(|C_v|\ge L)\ge (2+o(1))\eps.$
Equivalently, $\E_{n,p} (N_{[L,n]})\ge (2+o(1))\eps n$.
\end{lemma}
\begin{proof}
Taking $k=L$ in Lemma~\ref{couple}(ii),
\begin{eqnarray*}
\Pr_{n,p} (|C_v|\ge L) &\ge& \Pr(|\bp{n-L,p}|\ge L)\\
                 &\ge& \Pr (\bp{n-L,p} \ \text{survives})\sim 2\big( (n-L)p-1\big) \sim 2 \eps,
\end{eqnarray*}
where the approximation steps follow from \eqref{2eps} and the assumption on $L$.
\end{proof}

\begin{lemma}\label{eub}
Suppose that $L=L(n)$ satisfies $\eps^2L\to\infty$.
Then
$\E_{n,p} (N_{[L,n]})\le (2+o(1))\eps n$.
\end{lemma}
\begin{proof}
By Lemma~\ref{couple}(i) and \eqref{Tup},
\[
 \Pr_{n,p}(|C_v|\ge L) \le \Pr(|\bp{n,p}|\ge L) \le (1+o(1))(2\eps +1/(\eps L))\sim 2\eps.
\]
\end{proof}

Together these lemmas show that the expected number of vertices in components of size at
least $n^{2/3}$, say, is asymptotically $2\eps n$. Two tasks remain: to establish
concentration, and to show that most vertices in large components are in a single
giant component. For the first task, one can simply count tree components.
(This is a little messy, but theoretically trivial.
The difficulties in
the original papers~\cite{BB_evol,Luczak_near} stemmed from the fact that non-tree
components had to be counted as well. What is surprising is that here it suffices to count
tree components.) Indeed, applying the first
and second moment methods to the number $N$ of vertices in tree components
of size at most $n^{2/3}/\omega$, where $\omega=\omega(n)\to\infty$ sufficiently
slowly, shows that this number is within $\op(\epsilon n)$
of $(1-\rho)n$, reproving Lemma~\ref{eub} and (together with Lemma~\ref{elb})
giving the required concentration.
See~\cite{rg_bp} for a version of this argument with a (best possible) $\Op(\sqrt{n/\eps})$
error term.
Since the calculations, though requiring no ideas, are
somewhat lengthy, we take a different approach here.

\begin{lemma}\label{llarge}
Suppose that $L=L(n)$ satisfies $\eps^2L\to\infty$ and $L=o(\eps n)$. Then
\[
 N_{[L,n]}(G_{n,p}) = (2+\op(1))\eps n.
\]
\end{lemma}
\begin{proof}
Let $N$ be the number of vertices of $G_{n,p}$ in components
of size at least $L$. From Lemmas~\ref{elb} and~\ref{eub} the expectation $\E[N]$ of $N$
satisfies $\E[N]\sim 2\eps n$,
so it suffices to show that
\begin{equation}\label{aim}
 \E[N^2] \le (4+o(1))\eps^2n^2.
\end{equation}

Fix a vertex $v$ of $G_{n,p}$.
Let us reveal a tree $\cT_v'$ spanning a subset $C_v'$ of $C_v$ by exploring using
breadth-first search as before,
except that we stop the exploration if at any point (i) we have reached $L$ vertices
in total, or (ii) there are $\eps L$ vertices that have been reached (found as a new
neighbour of an earlier vertex) but not yet explored (tested for new neighbours).
Note that condition (ii) may happen partway through revealing a generation of $\cT_v'$,
and indeed partway through revealing the new neighbours of a vertex. We call a vertex
reached but not (fully) explored a \emph{boundary vertex}, and note
that there are at most $\eps L+1\le 2\eps L$ boundary vertices. Let $\cA$ be the event
that we stop for reasons (i) or (ii), rather than because we have revealed
the whole component:
\[
 \cA=\{\hbox{ the exploration stops due to (i) or (ii) holding }\}.
\]
Note that if $|C_v|\ge L$, then $\cA$ holds.

As before, we may couple $\cT_v'$ with $\bp{n,p}$ so that $\cT_v'\subset \bp{n,p}$.
Since the boundary vertices correspond to a set of vertices of $\bp{n,p}$ contained
in two consecutive generations, if $\cA$ holds, then either $|\bp{n,p}|\ge L$
or $w(\bp{n,p})\ge \eps L/2$. From \eqref{Tup} and \eqref{unwide} it follows that
$\Pr(\cA)\le (2+o(1))\eps$.

Since all vertices are equivalent and $|C_v|\ge L$ implies that $\cA$ holds, we have
\begin{equation}\label{eN2}
 \E[N^2] =n\E[1_{|C_v|\ge L}N] \le n \E[1_{\cA} N] = n\Pr(\cA) \E[N\mid \cA] \le (2+o(1))\eps n \E[N\mid \cA].
\end{equation}
Suppose that $\cA$ does hold. Given any vertex $w\notin C_v'$,
we explore from $w$ as usual, but within $G'=G_{n,p}\setminus V(C_v')$, coupling the resulting
tree $\cT_w'$ with $\bp{n,p}$ so that $\cT_w'\subset \bp{n,p}$. Let $C_w'$ be the component
of $w$ in $G'$, so $C_w'$ is spanned by $\cT_w'$.
Let $\cS$ be the event that (this final copy of) $\bp{n,p}$ is infinite, and let $\cD=\cS^\cc$.
Note that $C_w'\subset C_w$, and that the two are equal unless there is an edge from $C_w'$
to some boundary vertex. Since there are at most $2\eps L$ boundary vertices,
this last event has conditional probability at most $2\eps L |C_w'| p \le 3\eps L|C_w'|/n$, say.
Since $|C_w'|\le |\bp{n,p}|$, it follows that
\begin{eqnarray*}
 \Pr(|C_w|\ge L\mid \cA) &\le& \Pr(\cS) + \Pr(\cD)\Pr(|C_w'|\ge L \mid \cD) +3\Pr(\cD)\eps L n^{-1} \E[|C_w'| \mid \cD] \\
  &\le& \Pr(\cS) + \Pr(|\bp{n,p}|\ge L \mid \cD) +3\eps L n^{-1} \E[|\bp{n,p}| \mid \cD] \\
  &\le& \Pr(\cS) + (L^{-1} + 3\eps L n^{-1}) \E[|\bp{n,p}|\mid \cD],
\end{eqnarray*}
by Markov's inequality.
Since the final expectation above is $\sim \eps^{-1}$ and
our assumptions give that both $L^{-1}$ and $3\eps Ln^{-1}$ are $o(\eps^2)$, we see that
$ \Pr(|C_w|\ge L\mid \cA) \le (2+o(1))\eps$. Hence, recalling that
there are at most $L$ vertices in $C_v'$,
\[
 \E[N\mid \cA] \le L+(n-L)\Pr(|C_w|\ge L\mid \cA) \le L+(2+o(1))\eps n \sim 2\eps n.
\]
Combined with \eqref{eN2} this gives \eqref{aim}.
\end{proof}

To complete the proof of our main result,
it remains only to show that almost all vertices in large components are in a single
giant component. For this we use a simple form of the classical sprinkling argument
of Erd\H os and R\'enyi~\cite{ER_evol}.

\begin{proof}[Proof of Theorem~\ref{th1}]
It will be convenient to write $\eps=\omega n^{-1/3}$, with $\omega=\omega(n)\to\infty$ and
$\omega=o(n^{1/3})$. Also, let $\omega'\to\infty$ \emph{slowly},
say with $\omega' =o(\log \log \omega)$.

Set $L=\eps n/\omega'$.
By Lemma~\ref{llarge} there are in total at most $(2+\op(1))\eps n$ vertices
in components of size larger than $L$, which gives the upper bound on $L_1$.

For the lower bound, set $p_1=n^{-4/3}$, and define $p_0$ by $p_0+p_1-p_0p_1=p$, so that if first we choose
the edges with probability $p_0$ and then (we sprinkle some more) with probability $p_1$ then the random
graph we get is exactly $G_{n,p}$. Since $np_0-1=(1+o(1))\eps$, for any $\delta>0$
Lemma~\ref{llarge} shows that
with probability $1-o(1)$ the graph $G_{n,p_0}$ has at least $(2-\delta)\eps n$ vertices
in components of size at least $L$.

Let $U_1, \dots , U_{\ell}$ be the vertex sets of the components of $G_{n,p_0}$ of size at least $L$.
The probability that no edge sprinkled with probability $p_1$ joins $U_1$ to $U_j$ is
\[
(1-p_1)^{|U_1| |U_j|} \le e^{-p_1L^2} = \exp \big(-n^{-4/3} \omega^2 n^{4/3}/(\omega')^2\big),
\]
so the expected number of vertices of $U$ not contained in the component of $G_{n,p}$ containing
$U_1$ is at most
\[
 \sum_{j=2}^{\ell} \exp\big(-(\omega/\omega')^2\big) |U_j|=o(|U|).
\]
Consequently, with probability $1-o(1)$ all but at most $\delta|U|$ vertices of $U$ are contained
within a single component of $G_{n,p}$, in which case $L_1(G_{n,p})\ge (1-\delta)(2-\delta)\eps n$.
Since $\delta>0$ was arbitrary, it follows that $L_1(G_{n,p})\ge (2-\op(1))\eps n$,
completing the proof.
\end{proof}

To conclude,
let us remark that although Theorem~\ref{th1} is a key result about the phase
transition, as discussed in the introduction it is far from the final word
on the topic.

\end{document}